\documentclass[english, 12pt,leqno]{amsart}

\usepackage{amsmath}
\usepackage{amsfonts}
\usepackage{amssymb}
\newtheorem{Def}{D\'efinition}
\newtheorem{Thm}[Def]{Theorem}
\newtheorem{Prop}[Def]{Proposition}
\newtheorem{Lem}[Def]{Lemma}

\newtheorem{Cons}[Def]{Consequences}
\newtheorem{Exam}[Def]{Example}

\title[Malnormal subgroups]
{Malnormal subgroups and Frobenius groups:
basics and examples}

\author[Malnormal subgroups]
{Pierre de la Harpe and Claude Weber
\\
Appendix by Denis Osin}

\thanks{P.H. and C.W. are grateful to D.O. 
for several comments on the last-but-one version of this paper,
and for providing the appendix at the end of the paper.}

\address{Pierre de la Harpe and Claude Weber~:
Section de math\'ematiques, 
Universit\'e de Gen\`eve, 
C.P.~64, 
CH--1211 Gen\`eve 4. 
}
\email{Pierre.delaHarpe@unige.ch 
\hskip.1cm and \hskip.1cm
Claude.Weber@unige.ch}

\address{Denis Osin~:
Stevenson Center 1326, Department of Mathematics, Vanderbilt University Nashville,
TN 37240, U.S.A.
}
\email{denis.osin@gmail.com}


\date{April 10, 2011}

\begin{document}

\begin{abstract}
Malnormal subgroups occur in various contexts.
We review a large number of examples,
and we compare the situation in this generality to that 
of finite Frobenius groups of permutations.
\par
In a companion paper \cite{HaWe}, we analyse when
peripheral subgroups of knot groups and $3$-manifold groups are malnormal.
\end{abstract}

\maketitle

\section{\textbf{Introduction}
}
\label{section1}

A subgroup $H$ of a group $G$ is \emph{malnormal}
if $gHg^{-1} \cap H = \{e\}$ for all $g \in G$ with $g \notin H$.
As much as we know, 
the term goes back to a paper by Benjamin Baumslag
containing conditions for an amalgam $H \ast_L K$ 
(called a ``generalized free product'' in \cite{Baum--68})
to be $2$-free 
(= such that any subgroup generated by two elements is free).
Other authors write that $H$ is \emph{conjugately separated}
instead of ``malnormal'' \cite{MyRe--96}.
\par

The following question arose in discussions with Rinat Kashaev
(see also \cite{Kashaev} and \cite{Kash--11}).
We are grateful to him for this motivation.
\begin{itemize}
\item[]\emph{Given  a knot $K$ in $\mathbf S^3$,
when is the peripheral subgroup  malnormal
in the group $\pi_1(\mathbf S^3 \smallsetminus K)$ of $K$ ?
}
\end{itemize}
The answer, for which we refer to \cite{HaWe},
is that the peripheral subgroup is malnormal unless
$K$ is either a torus knot, or a cable knot, or a connected sum.

\medskip

The main purpose of the present paper is to collect 
in Section \ref{section3} several examples of pairs
\begin{center}
(infinite group, malnormal subgroup)
\end{center}
which are classical.
We recall  in Section \ref{section2} 
basic elementary facts on malnormal subgroups,
and we conclude in Section \ref{section4} by comparing 
the general situation with that of finite Frobenius groups.

\section{\textbf{General facts on malnormal subgroups}
}
\label{section2}

The two following propositions collect some straightforward properties
of malnormal subgroups.

\begin{Prop}
\label{Prop1}
Let $G$ be a group and $H$ a subgroup;
let $X$ denote the homogeneous space $G/H$
and let $x_0 \in X$ denote the class of $H$.
The following properties are equivalent:
\begin{itemize}
\item[(a)]
$H$ is malnormal in $G$;
\item[(b)]
the natural action of $H$ on $X \smallsetminus \{x_0\}$ is free;
\item[(c)]
any $g \in G$, $g \ne e$, has zero or one fixed point on $X$.
\end{itemize}
In case $G$ moreover contains a normal subgroup $N$ such that
$G$ is the semi-direct product $N \rtimes H$, these properties
are also equivalent to each of:
\begin{itemize}
\item[(d)]
$nh \ne hn$ for all $n \in N$, $n \ne e$, and $h \in H$, $h \ne e$;
\item[(e)]
$C_G(h) = C_H(h)$ for any $h \in H$, $h \ne e$. 
\end{itemize}
($C_G(h)$ denotes the centraliser $\{g \in G \hskip.1cm \vert \hskip.1cm gh=hg \}$
of $h$ in $G$.)
\end{Prop}

The proof is an exercise; if necessary, 
see the proof of Theorem 6.4 in \cite{Isaa--08}.

Following an ``added in proof'' of Peter Neumann in \cite{NeRo--98},
we define a \emph{Frobenius group} to be a group $G$
which has a malnormal subgroup $H$ distinct from $\{e\}$ and $G$.
A \emph{split Frobenius group} is a Frobenius group $G$
containing a malnormal subgroup $H$ and a normal subgroup $N$
such that $G = N \rtimes H$; then, it follows that the restriction to $N$
of the action of $G$ on $G/H$ is \emph{regular}, 
namely transitive with trivial stabilisers (the latter condition means
$\{ n \in N \hskip.1cm \vert \hskip.1cm n gH = gH \} = \{e\}$
for all $gH \in G/H$).

In finite group theory, according to a famous result of Frobenius,
Properties (a) to (c) \emph{imply} the existence 
of a splitting normal subgroup $N$,
so that any finite Frobenius group is split.
More on this in our Section \ref{section4}.

\begin{Prop}
\label{Prop2}
Let $G$ be a group.
\begin{itemize}
\item[(i)]
The trivial subgroups $\{e\}$ and $G$ are malnormal in $G$.
They are also the only subgroups of $G$ which are
both normal and malnormal.
\item[(ii)]
Let $H$ be a malnormal subgroup in $G$; 
then $gHg^{-1}$ is malnormal for any $g \in G$.
More generally $\alpha (H)$ is malnormal 
for any automorphism $\alpha$ of $G$.
\item[(iii)]
Let $H$ be a malnormal subgroup of $G$ and $K$ a malnormal subgroup of $H$;
then $K$ is malnormal in $G$.
\item[(iv)]
Let  $H$ be a malnormal subgroup and $S$ be a subgroup  of $G$;
then $H \cap S$ is malnormal in $S$.
\item[(v)]
Let $\left( H_{\iota} \right)_{\iota \in I}$ be a family of malnormal subgroups of $G$;
then $\bigcap_{\iota \in I} H_{\iota}$ is malnormal in $G$.
\item[(vi)]
Let $H$ and $H'$ be two groups;
then $H$ is malnormal in the free product $H \ast H'$.
\item[(vii)]
Let $H$ be a non--trivial subgroup of $G$;
if the centre $Z$ of $G$ is non--trivial, then $H$ is not malnormal in $G$.
\item[(viii)]
Let $H$ be a non-trivial subgroup of $G$ containing at least $3$ elements;
if $G$ contains a normal subgroup $C$ which is infinite cyclic,
then $H$ is not malnormal in $G$. 
\par
In particular, a group $G$ without $2$-torsion
containing a normal infinite cyclic subgroup
(such as the fundamental group of a Seifert manifold
not covered by $\mathbf S^3$)
does not contain any non-trivial malnormal subgroup.

\end{itemize}
\end{Prop}

\noindent 
\emph{Proof.}~ Claims (i) to (v) follow from the definition.
Claim (vi) follows from the usual normal form in free products,
and appears formally as Corollary 4.1.5 
of \cite{MaKS--66}.
\par
Claim (vi) carries over to amalgams $H \ast_K H'$
for $K$ malnormal in both $H$ and $H'$
.
\par

For (vii), we distinguish two cases.
First case: $Z \nsubseteq H$;
for $z \in Z$ with $z \notin H$, we have 
$z H z^{-1} \cap H = H \ne \{e\}$, so that~$H$ is not malnormal.
Second case: $Z \subset H$;
for $g \in G$ with $g \notin H$, we have
$\{e\} \ne Z \subset H \cap gHg^{-1}$.
\par

Claim (viii) is obvious if $H \cap C \ne \{e\}$, 
so that we can assume that $H \cap C = \{e\}$.
Choose $c \in C$, $c \ne e$.
For any $h \in H$, say with $h \ne e$, 
observe that $h^{-1}ch = c^{\pm 1}$.

If $h^{-1}ch = c$, then $e \ne h = c^{-1}hc \in H \cap c^{-1}Hc$,
and $H$ is not malnormal.
Since $H$ is not of order two, there exists $h_1, h_2 \in H \smallsetminus \{e\}$
with $k \Doteq h_2 h_1^{-1} \ne e$;
if $h_j^{-1} c h_j = c$ for at least one of $j = 1,2$, 
the previous argument applies;
otherwise $k^{-1}ck = c$,
so that $H$ is not malnormal for the same reason.
\hfill $\square$

 \medskip

About (viii), note that the infinite dihedral group $D_{\infty}$
contains an infinite cyclic subgroup of index $2$,
and that any subgroup of order $2$ in $D_{\infty}$ is malnormal.



\begin{Cons}
\label{Conseq}
Let $G$ be a group.
\begin{itemize}
\item[(ix)]
It follows from (v) that any subgroup $H$ of $G$
has a \emph{malnormal hull}\footnote{Or
\emph{malnormal closure}, as in Definition 13.5 of \cite{KaMy--02}}, 
which is the smallest malnormal subgroup of $G$ containing $H$.
\item[(x)]
It follows from (vi) that any group $H$ is isomorphic to
a malnormal subgroup of some group $G$.
\item[(xi)]
Let $\pi : G \longrightarrow Q$ be a projection onto a quotient group
and let $H$ be a malnormal subgroup in $G$;
then $\pi(H)$ need not be malnormal in $Q$.
\end{itemize}
\end{Cons}

For Claim (xi), consider a factor $\mathbf Z$ 
of the free product $G = \mathbf Z \ast \mathbf Z$
of two infinite cyclic groups,
and the projection $\pi$ of $G$ on its abelianization.
Then (xi) follows from (vi) and (vii).
\par

Consequence (ix)  above suggest the following construction,
potentially useful for the work of Rinat Kashaev.
Given a group $\mathcal G$ and a subgroup $\mathcal H$,
let $\mathcal N$ be the largest normal subgroup 
of $\mathcal G$ contained in $\mathcal H$;
set $H = \mathcal H / \mathcal N$, $G = \mathcal G / \mathcal N$.
Then $H$ has a malnormal hull, say $G_0$, in $G$.
There are interesting cases in which $H$ is malnormal in $G_0 = G$.
\par

For example, let $p,q$ be a pair of coprime integers, $p,q \ge 2$, 
and let $a,b \in \mathbf Z$ be such that $ap+bq = 1$.
Then $\mathcal G = \langle s,t \hskip.2cm \vert \hskip.2cm s^p = t^q \rangle$ 
is a torus knot group; a possible choice of meridian and parallel is
$\mu = s^b t^a$ and $\lambda = s^p \mu^{-pq}$, 
which generate the peripheral subgroup $\mathcal P$ of $\mathcal G$;
see e.g. Proposition 3.28 of \cite{BuZi--85}.
Let $\mathcal N = \langle s^p \rangle$ denote the centre of $\mathcal G$;
if $u$ and $v$ denote respectively 
the images of $s^b$ and $t^a$ in $G := \mathcal G / \mathcal N$, 
then $G  = \langle u, v \hskip.2cm \vert \hskip.2cm 
u^p = v^q = e \rangle \approx C_p \ast C_q$ is
the free product of two cyclic groups of orders $p$ and $q$,
and $P := \mathcal P / \mathcal N = \langle uv \rangle \approx \mathbf Z$
is the infinite cyclic group generated by $uv$.
As we check below (Example \ref{Gromhyp}.C), $P$ is malnormal in $G$.
\par

Let $G$ be a group and $H,H',K$ be subgroups such that
$K \subset H' \subset H \subset G$.
If $K$ is malnormal in $H$, then $K$ is malnormal in $H'$, by (iv).
It follows that it makes sense to speak about
maximal subgroups of $G$ in which $K$ is malnormal.
We will see in Example \ref{nonunique} that $K$ may be malnormal in \emph{several}
such maximal subgroups of $G$, none contained in any other;
in other words, one cannot define \emph{one} largest
subgroup of $G$ in which $K$ would be malnormal.

\section{\textbf{Examples of malnormal subgroups of infinite groups}
}
\label{section3}

Proposition \ref{Prop2} provides a sample of examples. 
Here are a few others.

\begin{Exam}[\textbf{Translation subgroup of the affine group of the line}]
\label{Exaff}
Let $\mathbf k$ be a field and let
$
G =
\left(\begin{matrix}
\mathbf k^{*} & \mathbf k
\\
0 & 1
\end{matrix}\right)
$
be its affine group,
where the subgroup $\mathbf k^{*}$ of $G$ is identified with
the isotropy subgroup of the origin
for the usual action of $G$ on the affine line $\mathbf k$.
Then $\mathbf k^{*}$ is malnormal in $G$.
\end{Exam}

Observe that  the field $\mathbf k$ need not be commutative
(in other words, $\mathbf k$ can be a division algebra).
More generally, if $V$ is a $\mathbf k$--module, 
then the subgroup
$
\mathbf k^{*} =
\left(\begin{matrix}
\mathbf k^{*} & 0
\\
0 & 1
\end{matrix}\right)
$
of the group
$
\mathbf k^{*} \ltimes V = 
\left(\begin{matrix}
\mathbf k^{*} & V
\\
0 & 1
\end{matrix}\right) 
$
is malnormal.
To check Example \ref{Exaff},
it seems appropriate to use (c) in Proposition~\ref{Prop1},
with $\mathbf k^{*} \ltimes V$ acting on $\mathbf k$ by
$((a,b), x) \longmapsto ax+b$.

\begin{Exam}[\textbf{Parabolic subgroup of a Fuchsian group}]
\label{ExFuch}
Consider a discrete subgroup
$\Gamma \subset PSL_2(\mathbf R)$ 
which is not elementary and which contains a parabolic element $\gamma_0$;
denote by $\xi$ the fixed point of $\gamma_0$ in the circle $\mathbf S^{1}$. 
Then the parabolic subgroup
\begin{equation*}
P \, = \, 
\{ \gamma \in \Gamma\hskip.1cm \vert \hskip.1cm \gamma(\xi) = \xi \}
\end{equation*}
corresponding to $\xi$ is malnormal in $\Gamma$.
\end{Exam}

For Example \ref{ExFuch}, 
we recall the  following standard facts.
The group $PSL_2(\mathbf R)$ is identified with the connected component
of the isometry group of the Poincar\'e half plane 
$\mathcal H^{2} = \{z \in \mathbf C \hskip.1cm 
\vert \hskip.1cm \operatorname{Im}(z) > 0 \}$;
it acts naturally on the boundary $\partial \mathcal H^{2} = \mathbf R \cup \{\infty\}$
identified with the circle $\mathbf S^{1}$.
Any $\gamma \in \Gamma$, $\gamma \ne e$,  
is either hyperbolic, with exactly two fixed points on 
$\mathbf S^{1}$,
or parabolic, with exactly one fixed point on $\mathbf S^{1}$,
or elliptic, without any fixed point on $\mathbf S^{1}$.
For such a group $\Gamma$ containing at least one parabolic element
fixing a point $\xi \in \mathbf S^2$, ``non-elementary" means $P \ne \Gamma$.
Since $\Gamma$ is discrete in $PSL_2(\mathbf R)$,
a point in the circle cannot be fixed by 
both a parabolic element and a hyperbolic element in $\Gamma$.
\par

It follows that the action of $P$ 
on the complement of $\{\xi\}$
in the orbit $\Gamma \xi$ satisfies Condition (b) 
of Proposition \ref{Prop1}, so that $P$ is malnormal in~$\Gamma$.

\par

In particular, in $PSL_2(\mathbf Z)$, 
the infinite cyclic subgroup $P$ generated by the class
$
\gamma_0 =
\left[\begin{matrix}
1 &1
\\
0 & 1
\end{matrix}\right]
$
of the matrix
$
\left(\begin{matrix}
1 &1
\\
0 & 1
\end{matrix}\right)
\in SL_2(\mathbf Z)
$
is malnormal (case of $\xi = \infty$).          
In anticipation of Section \ref{section4},
let us point out here that there \emph{cannot}
exist a normal subgroup $N$ of $PSL_2(\mathbf Z)$
such that $PSL_2(\mathbf Z) = N \rtimes P$,
because this would imply the existence of a surjection
$PSL_2(\mathbf Z) \longrightarrow P \approx \mathbf Z$,
but this is impossible since the abelianised group of $PSL_2(\mathbf Z)$
is finite (cyclic of order $6$).

\begin{Exam}[\textbf{Parabolic subgroup of a torsion--free Kleinian group}]
\label{ExKlei}
Consider a  discrete subgroup
$\Gamma \subset PSL_2(\mathbf C)$ 
which is not elementary, 
which is torsion--free, 
and which contains a parabolic element $\gamma_0$;
denote by $\xi$ the fixed point of $\gamma_0$ in the sphere $\mathbf S^{2}$. 
Then the parabolic subgroup
\begin{equation*}
P \, = \, 
\{ \gamma \in \Gamma \hskip.1cm \vert \hskip.1cm \gamma(\xi) = \xi \}
\end{equation*}
corresponding to $\xi$ is malnormal in $\Gamma$.
\end{Exam}

The argument indicated for Example \ref{ExFuch} 
carries over to the case of Example~\ref{ExKlei}.
The group $PSL_2(\mathbf C)$ is identified with the connected component
of the isometry group of the hyperbolic $3$--space 
$\mathcal H^{3} = \{(z,t) \in \mathbf C  \times \mathbf R 
\hskip.1cm \vert \hskip.1cm  t > 0 \}$;
it acts naturally on the boundary 
$\partial \mathcal H^{3} = \mathbf C \cup \{\infty\}$
identified with the sphere $\mathbf S^{2}$.

\par
The number of conjugacy classes of subgroups of the type $P$
is equal to the number of orbits of $\Gamma$
on the subset of $\mathbf S^{2}$ consisting of points
fixed by some parabolic element of $\Gamma$.
\par
In Example~\ref{ExKlei}, 
the hypothesis that $\Gamma$ is torsion--free cannot be deleted.
For example, consider the Picard group $PSL_2(\mathbf Z[i])$, 
its subgroup $P$ of classes of matrices of the form
$
\left[ \begin{matrix}
a &b
\\
0 & a^{-1}
\end{matrix} \right] 
$,
and the subgroup $Q$ of $P$ of classes of matrices of the form
$
\left[ \begin{matrix}
1 & b
\\
0 & 1
\end{matrix}\right] .
$
Set
\begin{equation*}
g \, = \, 
\left[ \begin{matrix}
\phantom{-}0 &1
\\
-1 & 0
\end{matrix} \right] \, \in \, PSL_2(\mathbf{Z}[i])
\hskip.5cm (\text{observe that} \hskip.2cm
g^2 = e, \hskip.1cm g \notin P),
\end{equation*}
and
\begin{equation*}
h \, = \, 
\left[ \begin{matrix}
i & \phantom{-}0
\\
0 & -i
\end{matrix}\right] \, \in \, P
\hskip.5cm (\text{observe that} \hskip.2cm
h^2 = e, \hskip.1cm h \notin Q).
\end{equation*}
As $ghg^{-1} = h^{-1} \in P$, the subgroup $P$ 
is not malnormal in $PSL_2(\mathbf Z[i])$).
As 
$h 
\left[ \begin{matrix}
1 & b
\\
0 & 1
\end{matrix}\right]
h^{-1} \in Q
$
for all $b$, the subgroup $Q$ is not malnormal in $P$
(and \emph{a fortiori} not malnormal in $PSL_2(\mathbf Z[i])$).
\par

Note that, in the group $\Gamma$ of
Example \ref{ExFuch}, torsion is allowed,
because each element $\gamma \ne e$ of finite order
in $PSL_2(\mathbf R)$ acts without fixed point on $\mathbf S^1$.
But the element $h$ above, of order $2$, 
has fixed points on $\mathbf S^2$.

\medskip

\noindent 
\emph{Generalisation.}
In case the group $\Gamma$ of Example \ref{ExKlei}
is the group of a hyperbolic knot, 
$P$ is the \emph{peripheral subgroup} of $\Gamma$.
This carries over to a much larger setting, see Lemma \ref{malnL} below.
This itself can be extended to ``peripheral subgroups''
in much more general situations, 
the conclusion being then that these subgroups are \emph{almost} malnormal
\cite[Theorem 1.4]{Osi--06b}.
\par
A subgroup $H$ of a group $G$ is \emph{almost malnormal}
if $gHg^{-1} \cap H$ is finite for any $g \in G$ with $g \notin H$,
equivalently if the following condition holds:
for any pair of distinct points $x,y \in G / H$, 
the subgroup $\{g \in G \hskip.1cm \vert \hskip.1cm gx=x, gy=y \}$
is finite.

\medskip

Example \ref{Gromhyp} is a variation on Examples \ref{ExFuch} and  \ref{ExKlei},
related to boundary fixed points of hyperbolic elements
rather than of parabolic elements.

\begin{Exam}
[\textbf{Virtually cyclic subgroups of torsion free Gromov hyperbolic groups}]
\label{Gromhyp}
Con\-sider a  Gromov hyperbolic group
$\Gamma$ 
which is  not elementary, 
an element $\gamma_0 \in \Gamma$ of infinite order,
and one of the two points in the boundary $\partial \Gamma$
fixed by $\gamma_0$, say  $\xi$.
Set
\begin{equation*}
P \, = \, 
\{ \gamma \in \Gamma \hskip.1cm \vert \hskip.1cm \gamma(\xi) = \xi \} .
\end{equation*}
Assume moreover that 
\begin{equation}
\text{any $\gamma \in \Gamma \smallsetminus \{e\}$ of finite order
acts without fixed point on $\partial \Gamma$}
\end{equation}
(this is trivially the case
if $\Gamma$ is torsion free).
\par
  Then $P$ is malnormal in $\Gamma$.
\end{Exam}

For the background of Example \ref{Gromhyp},
see  \cite{GhHa--90}, 
in particular Theorem 30 of Chapter 8.
Recall that any element in $P$ fixes also 
the other fixed point of $\gamma_0$,
and that the infinite cyclic subgroup of $\Gamma$
generated by $\gamma_0$ is of finite index in $P$.

\medskip

7.A.
Here is an illustration of Example \ref{Gromhyp}: 
in the free group $F_2$ on two generators $a$ and $b$,
any primitive element, for example $a^kba^{\ell}b^{-1}$
with $k,l \in \mathbf Z \smallsetminus \{0\}$,
generates an infinite cyclic subgroup which is malnormal.
(An element $\gamma$ in  a group $\Gamma$ is \emph{primitive}
if there does not exist any pair $(\delta,n)$,
with $\delta \in \Gamma$ and $n \in \mathbf Z$, $\vert n \vert \ge 2$,
such that $\gamma = \delta^{n}$.)

\medskip

7.B.
In torsion free non-elementary hyperbolic groups,
subgroups of the form $P$ are precisely the
maximal abelian subgroups, which are malnormal. 
(See also Example \ref{CSA}.)

\medskip

7.C.
Here is an old-fashioned variation on Example \ref{Gromhyp}:
consider a discrete subgroup $\Gamma$ of $PSL_2(\mathbf R)$,
and let $h \in \Gamma$ be a hyperbolic element fixing two distinct points
$\alpha, \omega \in \mathbf S^1$;
then 
\begin{equation*}
H \, = \,  \{ \gamma \in \Gamma \hskip.1cm \vert \hskip.1cm 
\gamma(\alpha) = \alpha \}
\end{equation*} 
is malnormal in $\Gamma$.
\par

For an illustration (referred to at the end of Section \ref{section2}), 
consider two integers $p \ge 2$ and $q \ge 3$; 
denote by $C_p$ and $C_q$ the finite cyclic groups of order $p$ and $q$.
It is easy to show\footnote{In the hyperbolic plane,
consider a rotation $u$ of angle $2\pi/p$ 
and a rotation $v$ of angle $2\pi/q$. 
If the hyperbolic distance between the fixed points of $u$ and $v$ is large enough,
the group generated by $u$ and $v$ is a free product $C_p \ast C_q$
(by the theorem of Poincar\'e 
on polygons generating Fuchsian groups \cite{Rham--71}),
and the product $uv$ is hyperbolic.} 
that there exists a non-elementary Fuchsian group $\Gamma$
generated by two isometries $u,v$ such that
$\Gamma = \langle u,v \hskip.2cm \vert \hskip.2cm u^p = v^q = e \rangle 
\approx C_p \ast C_q$,
in which the product $uv$ is hyperbolic and primitive.
It follows that
the infinite cyclic group generated by $uv$ is malnormal in $\Gamma$.

\medskip

7.D.
Let us mention another variation:
let $\mathbb G$ be a connected semisimple real algebraic group 
without compact factors,
let $d$ denote its real rank, and let $\Gamma$ be a torsion free uniform lattice
in $G := \mathbb G(\mathbf R)$.
Then $\mathbb G$ contains a maximal torus $\mathbb T$ such that
$A := \Gamma \cap \mathbb T (\mathbf R) \approx \mathbf Z^d$
is malnormal in $\Gamma$;
see \cite{RoSt--10}, building up on a result of Prasad and Rapinchuk,
and motivated by the construction of an example in operator algebra theory.
For an earlier use of malnormal subgroups in operator algebra theory,
see \cite{Robe--06}, in particular Corollary 4.4, 
covered by our Example \ref{Gromhyp}.

\medskip

7.E
Consider the situation of Example \ref{Gromhyp}
\emph{without} the hypothesis (1).
Then it is still true that $P$ is \emph{almost malnormal}
(as defined in the generalisation of Example \ref{ExKlei}).

\medskip

The following example supports the last claim of Section \ref{section2}.

\begin{Exam}
\label{nonunique}
There exists a group $G$ containing two distinct maximal subgroups $B_+, B_-$
and a subgroup $T \subset B_+ \cap B_-$
which is malnormal in each of $B_+, B_-$, but not in $G$.
\end{Exam}

Set  $G = PGL_2(\mathbf C)$; let
$\pi : GL_2(\mathbf C) \longrightarrow G$ denote the canonical projection.
Define the subgroups
\begin{equation*}
T \, = \, \pi 
\left( \begin{matrix}
\mathbf C^{*} & 0
\\
0 & \mathbf C^{*}
\end{matrix}\right) ,
\hskip.3cm
B_+ \, = \, \pi 
\left( \begin{matrix}
\mathbf C^{*} & \mathbf C
\\
0 & 1
\end{matrix}\right) ,
\hskip.3cm \text{and} \hskip.3cm 
B_- \, = \, \pi 
\left( \begin{matrix}
\mathbf C^{*} & 0
\\
\mathbf C & 1
\end{matrix}\right)
\end{equation*}
of $G$.
Then $T$ is malnormal in $B_+$ and in $B_-$
(see Example~\ref{Exaff} and Proposition  \ref{Prop2}.ii), but not in $G$,
since $T$ is strictly contained in its normalizer $N_G(T)$,
the quotient $N_G(T)/T$ being the Weyl group 
of order $2$.
Moreover, $B_+$ and $B_-$ are maximal subgroups in $G$;
this can be checked in an elementary way,
and is also a consequence of 
general properties of parabolic subgroups
(see \cite{Bour--68}, Chapter~IV, $\S$~2, n$^{o}$ 5, Th\'eor\`eme 3).

\begin{Exam}[\textbf{CSA}]
\label{CSA}
A group is said to be CSA if all its maximal abelian subgroups are malnormal.
\par
The following groups are known to be CSA :
(i) torsion free hyperbolic groups, 
(ii) groups acting freely and without inversions on $\Lambda$-trees (in particular on trees),
and (iii) universally free groups.\end{Exam}

For (i), see Example \ref{Gromhyp}.B. 
For (ii), see \cite{Bass--91}, Corollary 1.9;
(iii) follows.
For CSA groups, see 
\cite{MyRe--96},  \cite{GiKM--95},
and other papers by the same authors.


\medskip

The nature of the two last examples of this section
is more combinatorial than geometric.
 
\begin{Exam}
[\textbf{M.~Hall}]
\label{??}
Let $F$ be a free group and $H$ a finitely generated subgroup of $F$;
then there exist a subgroup of finite index $F_0$ in $F$
which contains $H$ and a subgroup $K$ of $F$
such that $F_0 = H \ast K$.

In particular, $H$ is malnormal in $F_0$.
\end{Exam}

This is a  result due to M.~Hall and often revisited.
See \cite{Hall--49}, \cite{Burn--69}, \cite{Stal--83},
or Lemma 15.22 on Page 181 of \cite{Hemp--76}.
\par
Rank $2$ malnormal subgroups of free groups are characterised in
\cite{FiMR--02}.


\begin{Exam}
[\textbf{An example of B.B.~Newman}]
\label{BBN}
Let $G = \langle X ; r \rangle$ be a one relator group with torsion,
let $Y$ be a subset of $X$ which omits
at least one generator occuring in $r$,
and let $H$ be the subgroup of $G$ generated by $Y$.
Then $H$ is malnormal in $G$.
\end{Exam}

For Example \ref{BBN}, see Chapter IV of
\cite{LySc--77}, just before Theorem  5.4, Page 203.

\medskip

To conclude this section, we quote two more known facts
about malnormal subgroups:
\begin{itemize}
\item[$\bullet$]
there exist hyperbolic groups for which there is no algorithm
to decide which finitely generated subgroups are malnormal
\cite{BrWi--01};
\item[$\bullet$]
if $H$ is a finitely generated subgroup 
of a finitely generated free group $F$,
the malnormal closure $K$ of $H$ in $F$ 
has been investigated in \cite{KaMy--02};
in particular, $K$ is finitely generated
(part of Theorem 13.6 in \cite{KaMy--02}).
\end{itemize}
Other papers on malnormal groups include
\cite{BaMR--99}
and \cite{KaSo--71}.

\medskip

Almost malnormal subgroups have hardly be mentioned here
(but in the generalisation of Example \ref{ExKlei}).
They have nevertheless their importance,
for example for proving residual finiteness of some groups
in various papers by Daniel T. Wise (three of them quoted in our references).
Here is a result of \cite{Wis--02a}:
\emph{the free product of two virtually free groups
amalgamating a finitely generated almost malnormal subgroup
is residually finite}. The malnormality condition is necessary!
indeed: \emph{there exists a free group $F$ and a subgroup $E$ of finite index
such that the amalgamated product $F \ast_E F$ is
an infinite simple group} \cite[Theorem 5.5]{BuMo--00}.

\section{\textbf{Comparison with malnormal subgroups of finite groups}
}
\label{section4}

Let $H$ be a malnormal subgroup in a  group $G$, with $H \ne \{e\}$ and $H \ne G$.
Let $N$ denote the \emph{Frobenius kernel}, which is by definition
the union of $\{e\}$ and of the complement in $G$
of $\bigcup_{g \in G} g H g^{-1}$;
observe that $N \smallsetminus \{e\}$ is the set of elements in $G$
without any fixed point on $G/H$.
The subgroup $H$ is called the \emph{Frobenius complement}.
\par

For the case of a \emph{finite group} $G$, 
let us quote the following three important results,
for which we refer to Theorems V.7.6, V.8.7, and V.8.17 in \cite{Hupp--67};
see also \cite{Asch--00} 
or \cite{DiMo--96}.
\begin{itemize}
\item[$\circ$]
\emph{$N$ is a normal subgroup of $G$,
its size $\vert N \vert$ coincides with the index $[G:H]$,
and $G$ is a semi-direct product $N \rtimes H$ (Frobenius \cite{Frob--01});
moreover $\vert N \vert \equiv 1 \pmod{\vert H \vert}$.}
%
%
\item[$\circ$]
\emph{$N$ is a nilpotent group; moreover, 
if $H$ is of even order, then $N$ is abelian
(Thompson \cite{Thom--59}, \cite{Thom--60});}
\item[$\circ$]
Let $H'$ be another malnormal subgroup in $G$, neither $\{e\}$ nor $G$,
and let $N'$ be the corresponding complement;
then $H'$ is conjugate to $H$ and $N' = N$.
\end{itemize}
(Moreover, $N'=N$ coincides with the ``Fitting subgroup'' of $G$,
namely the largest nilpotent normal subgroup of $G$.)
There are known, but non-trivial, examples 
showing that $H$ need not be solvable, and that $N$ need not be abelian.

\medskip

These facts do not carry over to infinite groups,
as already noted in several places including 
\cite{Coll--90} and Page 90 in \cite{DiMo--96}.
In what follows and as usual, $G$ is a group with malnormal subgroup $H$
and Frobenius complement $N$; we assume that $H$ is neither $\{e\}$ nor $G$.

\subsection{$N$ need not be a subgroup of $G$.}
For example, let $K$ be a non-trivial knot in $\mathbf S^3$,
$G_K$ its group, and $P_K$ its peripheral subgroup.
Assume that $K$ is prime, and neither a torus knot nor a cable knot,
so that $P_K$ is malnormal in $G_K$ \cite{HaWe}.
Since the abelianisation of $G_K$ is $\mathbf Z$
(for example by Poincar\'e duality),
the Frobenius kernel is not a subgroup
(otherwise $P_K \approx \mathbf Z^2$ would be 
a quotient of $G_K^{\text{ab}} \approx \mathbf Z$, which is preposterous).
\par

Another example is provided by $H$, malnormal in $G = H \ast K$,
with $H$ and $K$ non-trivial and not both of order $2$
(see Proposition \ref{Prop2}.vi).
Again, the kernel is not a subgroup; indeed,
for $h_1,h_2 \in H \smallsetminus \{e\}$ with $h_1h_2 \ne e$ and $k \in K \smallsetminus \{e\}$,
then $h_1k$ and $k^{-1}h_2$ are in the complement, but $h_1h_2$ is not.
\par

The example of the cyclic subgroup $H$ generated by $x^{-1}y^{-1}xy$
in the free group $G$ on two generators $x$ and $y$, which is a malnormal subgroup
of which the complement is not a subgroup, appears on Page 51 of \cite{KeWe--73};
see also  Example \ref{Gromhyp}.A above.

\subsection{There are examples with $N = \{e\}$.}
Consider a large enough prime $p$ and a \emph{Tarski monster} for $p$,
namely an infinite group $G$ in which any non-trivial subgroup is cyclic of order $p$.
Such subgroups have been shown to exist by  Ol'shanskii (1982),
see \S~28 in \cite{Ol's--91}, 
and independently by Rips (unpublished, cited in \cite{Coll--90});
note that such a $G$ is necessarily generated by two elements, and is a simple group.
Then $G$ acts by conjugation on the set $X$ of its non-trivial subgroups,
in a transitive way.
This makes $G$ a Frobenius group of permutations, since any $g \in G$, $g \ne e$
has a unique fixed point in $X$ which is $\{e, g, g^2, \hdots, g^{p-1}\}$.
Thus $N$ is reduced to $\{e\}$, and $G$ is certainly not a semi-direct product
of $N$ and a cyclic group of order $p$.
This example has been noted in several places, one being \cite{Came--86}.

\subsection{When $N$ is a subgroup of $G$, it need not be nilpotent}
This is shown by the example of the wreath product $G = S \wr \mathbf Z$,
with $S$ a simple group.
The subgroup $H = \mathbf Z$ is malnormal, and the corresponding Frobenius kernel
$N = \bigoplus_{i \in \mathbf Z} S_i$, with each $S_i$ a copy of $S$, 
is not nilpotent.
More generally, given a group $H$ acting on a set $X$
in such a way that $h^{\mathbf Z}x$ is infinite for all $h \in H$, $h \ne e$, and $x \in X$,
as well as a group $S \ne \{e\}$, 
the permutational wreath product $G = S \wr_X H$ contains $H$ as a malnormal subgroup,
with Frobenius kernel $\bigoplus_{x \in X} S_x$.

\subsection{Malnormal subgroups need not be conjugate.}
This is clear with a free product $G = H \ast K$ as in Proposition \ref{Prop2},
where $H$ and $K$ are both malnormal subgroups, and are clearly non-conjugate.
If $G = H \ast K \ast L$, with non-trivial factors, 
the malnormal subgroup $H$ is strictly contained in the malnormal subgroup $H \ast K$.

\subsection{A last question, out of curiosity}
\label{curiosity}
Let $G = N \rtimes H$ be a semi-direct product.
If $X = G/H$ (as in Proposition \ref{Prop1}) is identified with $N$,
note that the natural action of $G$ can be written like this:
$g = mh \in G$ acts on $n \in N$ to produce $mhnh^{-1} \in N$.
Consider the two following conditions, the first being as in Proposition \ref{Prop1}:
\begin{itemize}
\item[(a)]
$H$ is malnormal in $G$;
\item[(f)]
$C_G(n) = C_N(n)$ for any $n \in N$, $n \ne e$. 
\end{itemize}
Then (f) implies (a). Indeed, for any $n \in N$, Condition (f) implies
$C_H(n) = H \cap C_G(n) \subset H \cap N = \{e\}$;
in other terms, for $h \in H$, $h \ne e$, 
the equality $hnh^{-1} = n$ implies $h = e$.
Thus Condition (b) of Proposition \ref{Prop1} holds, and thus (a) holds also.

When $G$ is finite, then, conversely,  (a) implies (f);
see Theorem 6.4 in \cite{Isaa--08}. 
Does this carry over to the general case?

\medskip

As Denis Osin has answered this question negatively, we reproduce his argument below.

\section{\textbf{Appendix by Denis Osin
\\ 
Answer to the question of \ref{curiosity}}}

Observe that for any split extension $G=N\rtimes H$ 
and for any element $z\in N$, 
$C_N(z)$ is normal in $C_G(z)$ and $C_G(z)/C_N(z)$ 
is isomorphic to a subgroup of $H$. 
It turns out that \emph{every} subgroup of $H$ can be realized as $C_G(z)/C_N(z)$ 
for some $z\in N$ and $G=N\rtimes H$ with $H$ malnormal. 
Below we prove this for torsion free finitely generated groups. 
The proof of the general case is a bit longer. 
It is based on the same idea but uses some additional technical results 
about van Kampen diagrams and small cancellation quotients of relatively hyperbolic groups.

\begin{Thm}
\label{mainT}
For any finitely generated torsion free group $H$ 
and any finitely generated subgroup $Q\le H$,
there exists a group $N$, a split extension $G=N\rtimes H$, 
and a nontrivial element $z\in N$ 
such that $H$ is malnormal in $G$ and $C_G(z)/C_N(z)\cong Q$.
\end{Thm}

To prove the theorem we will need some tools 
from small cancellation theory over relatively hyperbolic groups. 
Let $G$ be a group hyperbolic relative 
to a collection of subgroups $\{ H_\lambda \} _{\lambda \in \Lambda }$. 
An element of $G$ is \emph{loxodromic} 
if it is not conjugate to an element of 
$\bigcup\limits_{\lambda \in \Lambda } H_\lambda$ 
and has infinite order. 
A group is \emph{elementary} if it contains a cyclic subgroup of finite index. 
It is proved in \cite{Osi--06a} that for every loxodromic element $g\in G$, 
there is a unique maximal elementary subgroup $E_G(g)\le G$ containing $g$. 
Two loxodromic elements $f,g$ are \emph{commensurable} (in G) if
$f^k$ is conjugate to $g^l$ in $G$ for some non--zero $k,l$. 
A subgroup $S$ of $G$ is called \emph{suitable} 
if it contains two non-commensurable loxodromic elements $g,h$ 
such that $E_G(g)\cap E_G(h)=\{ 1\} $.

The lemma below follows immediately from \cite[Corollary 2.37]{Osi--06b}. 
\begin{Lem}
\label{malnL}
Let $G$ be a  torsion free group hyperbolic relative 
to a collection of subgroups $\{ H_\lambda \} _{\lambda \in \Lambda } $. 
Then every $H_\lambda $ is malnormal.
\end{Lem}

The next result follows from \cite[Theorem 2.4]{Osi--10} and its proof.

\begin{Thm}
\label{suit}
Let $G_0$ be a group hyperbolic relative to 
a collection $\{ H_\lambda \} _{\lambda \in \Lambda }$, 
and $S$ a suitable subgroup of $G_0$. 
Then for every finite subset $T\subset G_0$, there exists a set of elements $\{ s_t\mid t\in T\} \subset S$ such that the following conditions hold.
\begin{enumerate}
\item[(a)] 
Let $G=\langle G_0 \mid t=s_t, t \in T\rangle $. 
Then the restriction of the natural homomorphism $\varepsilon \colon G_0\to G$ 
to every $H_\lambda $ is injective.
\item[(b)] 
$G$ is hyperbolic relative to 
$\{ \varepsilon (H_\lambda )\}_{\lambda \in \Lambda}$.
\item[(c)] 
If $G_0$ is torsion free, then so is $G$.
\end{enumerate}
\end{Thm}

\begin{proof}[Proof of Theorem \ref{mainT}]
Let 
\begin{equation*}
G_0=(\langle z\rangle \times Q)\ast \langle a, b\rangle \ast H.
\end{equation*}
Clearly $G_0$ is hyperbolic relative to the collection 
$\{ \langle z \rangle \times Q , H\}$. 
Let $X$ and $Y$ be finite generating sets of $Q$ and $H$, respectively. 
We fix an isomorphic embedding $\iota \colon Q\to H$. 
Without loss of generality we can assume that $\iota (X)\subseteq Y$.

It is easy to see that $S=\langle a,b\rangle $ is a suitable subgroup of $G_0$. 
Indeed $a$ and $b$ are not commensurable in $G_0$ 
and $E_{G_0} (a)\cap E_{G_0}(b)=\{ 1\}$. 
We apply Theorem \ref{suit} to the finite set
\begin{equation*}
T=\{ z\} \cup \{ a^y, b^y \mid y\in Y\cup Y^{-1}\} \cup \{ x^{-1}\iota (x) \mid x\in X\}.
\end{equation*}
Let $G$ be the corresponding quotient group. 
For simplicity we keep the same notation for elements of $G_0$ 
and their images in $G$. 
Part (a) of Theorem \ref{suit} also allows us 
to identify the subgroups $\langle z\rangle \times Q $ and $H$ of $G_0$ 
with their (isomorphic) images in $G$.

Let $N$ be the image of $S$ in $G$. 
Note that, in the quotient group $G$, we have $t\in N$ for every $t\in T$. 
In particular we have $z\in N$ and $x^{-1}\iota (x)\in N$ for all $x\in X$. 
Hence the group $G$ is generated by $\{ a, b\}\cup Y$. 
Since $a^y, b^y \in N$ for all $y\in Y\cup Y^{-1}$, 
the subgroup $N$ is normal in $G$ and $G=NH$. 
Using Tietze transformations it is easy to see 
that the map $a\mapsto 1$ and $b\mapsto 1$ 
extends to a retraction $\rho \colon G\to H$ 
such that $\rho\vert _Q\equiv \iota $. 
In particular, $H\cap N=\{e\}$ and hence $G=N\rtimes H$.

Since $G_0$ is torsion free, so is $G$ by Theorem \ref{suit} (c). 
By Theorem \ref{suit} (b) and Lemma \ref{malnL}, 
the subgroups $H$ and $\langle z\rangle \times Q$ are malnormal in $G$. 
In particular, $C_G(z)  = \langle z\rangle \times Q$. 
Since $z\in N$ and $\rho\vert _Q\equiv \iota $, 
we obtain $\rho (z^n, q)=\iota (q)$ for every $(z^n, q)\in \langle z\rangle \times Q$. Hence
\begin{equation*}
C_N(z)=C_G(z)\cap N=C_G(z)\cap {\rm Ker }\, \rho = \langle z\rangle \times \{e\} .
\end{equation*}
Therefore $C_G(z)/C_N(z)\cong Q$.
\end{proof}

\end{document}